\documentclass[11pt]{amsart}

\usepackage{color}
\usepackage{amsmath}
\usepackage{amsfonts}
\usepackage{amssymb,enumerate}
\usepackage{enumitem}
\usepackage{amsthm}
\usepackage{algorithm}
\usepackage{algpseudocode}
\usepackage{color}
\usepackage{comment}
\usepackage[all]{xy}
\usepackage[utf8]{inputenc}
\usepackage[noadjust]{cite}
\usepackage{lineno}
\usepackage{listings}
\usepackage{tikz}
\usepackage{quiver}
\usepackage{lineno}
\usepackage{makecell}
\usepackage[final]{pdfpages}
\usepackage{cellspace} 
\theoremstyle{definition}
\newtheorem{lem}{Lemma}[section]
\newtheorem{cor}[lem]{Corollary}

\newtheorem{prop}[lem]{Proposition}
\newtheorem{thm}[lem]{Theorem}
\newtheorem{alg}[lem]{Algorithm}

\newtheorem{defn}[lem]{Definition}
\newtheorem{ex}[lem]{Example}

\usepackage[margin=1.3in]{geometry}

\renewcommand{\phi}{\varphi}

\newcommand{\Aut}{\text{Aut}}

\newcommand{\setx}{\langle x \rangle}
\newcommand{\sety}{\langle y \rangle}

\renewcommand{\geq}{\geqslant}
\renewcommand{\leq}{\leqslant}

\newenvironment{breakablealgorithm}
  {
   \begin{center}
     \refstepcounter{algorithm}
     \hrule height.8pt depth0pt \kern2pt
     \renewcommand{\caption}[2][\relax]{
       {\raggedright\textbf{\fname@algorithm~\thealgorithm} ##2\par}%
       \ifx\relax##1\relax 
         \addcontentsline{loa}{algorithm}{\protect\numberline{\thealgorithm}##2}%
       \else 
         \addcontentsline{loa}{algorithm}{\protect\numberline{\thealgorithm}##1}%
       \fi
       \kern2pt\hrule\kern2pt
     }
  }{
     \kern2pt\hrule\relax
   \end{center}
  }
  
\numberwithin{equation}{section}

\begin{document}
\address{PRIMES-USA}
\email[A.~Agarwal]{arjunagarwal010@gmail.com}
\email[R.~Chen]{rachelrxchen@gmail.com}
\email[R.~Garg]{rohangarg2008@gmail.com}
\email[J.~Kettinger]{jkettin@g.clemson.edu}
\author{Arjun Agarwal}
\author{Rachel Chen}
\author{Rohan Garg}
\author{Jared Kettinger}

\title{Automorphically Equivalent Elements  Of Finite Abelian Groups}

\begin{abstract}
Given a finite abelian group $G$ and elements $x, y \in G$, we prove that there exists $\phi \in \Aut(G)$ such that $\phi(x) = y$ if and only if $G/\setx \cong G/\sety$. This result leads to our development of the two fastest known algorithms to determine if two elements of a finite abelian group are automorphic images of one another. The second algorithm also computes $G/\langle x \rangle$ in a near-linear time algorithm for groups, most feasible when the group has exponent at most $10^{20}$. We conclude with an algorithm that computes the automorphic orbits of finite abelian groups.
\end{abstract}
\maketitle

\sloppy
\section{Introduction}

In this paper, we study a necessary and sufficient condition for two elements of a finite abelian group $G$ to be in the same automorphic orbit and use this to develop multiple algorithms. We prove that elements $x$ and $y$ in finite abelian group $G$ are in the same automorphic orbit if and only if $G/\setx \cong G/\sety$. We further devise algorithms to check this condition, using them to classify and count the elements in each orbit. The ease of computation of $G/\setx$ leads us to develop the two fastest known algorithms for computing whether two elements of a finite abelian group are automorphic images of each other and the fastest known algorithm for computing the automorphism orbits of finite abelian groups. 

The topic of automorphisms in finite abelian groups has been extensively studied, but previous works were focused on describing and counting the automorphisms (see \cite{hillar}) or counting the number of orbits (see \cite{SS, dutta}). Our work builds on these works, examining other areas of automorphisms in finite abelian groups as well. Throughout this paper, we place an emphasis on the computational speed of our algorithms, an area that has not been extensively investigated in the study of automorphisms in finite abelian groups. 

\section{Condition for Automorphic Equivalence}
For the rest of the paper, let $G$ denote a finite abelian group. 
\begin{thm}\label{mainres}
    Given $x$, $y \in G$, there exists $\phi \in \Aut(G)$ such that $\phi(x) = y$ if and only if $G/\setx \cong G/\sety$. 
\end{thm}

\begin{lem}\label{lemmaprod}
(Lemma 2.1 in \cite{hillar})
If $H$ and $K$ are finite groups with relatively prime orders, 
$$
\Aut(H) \oplus \Aut(K) \cong \Aut(H \oplus K). 
$$
\end{lem}

\begin{prop}\label{prop:pgroup}
    In order to prove that $G/\setx \cong G/\sety$ implies there exists $\phi \in \Aut(G)$ such that $\phi(x) = y$, it is sufficient to consider $p$-groups. 
\end{prop}
\begin{proof}
We know that any finite abelian group is isomorphic to the direct sum of its
Sylow $p$-subgroups. Therefore, if $G$ is a finite abelian group such that $|G| = {p_1}^{e_1}{p_2}^{e^2}\dots{p_n}^{e_n}$ where each $p_i$ is a distinct prime, then 
$$
G \cong \bigoplus_{i=1}^n H_{p_i},
$$
where $H_{p_j}$ is the unique Sylow $p_j$-subgroup.
We will represent elements of $G$ as $n$-tuple with the $i^\text{th}$ member of the tuple being an element of $H_{p_i}$. 
For $x,y \in G$,
let $x = (x_{p_1},x_{p_2},\dots,x_{p_n})$ and
$y = (y_{p_1},y_{p_2},\dots,y_{p_n})$, where $x_{p_i}, y_{p_i} \in H_{p_i}$ for each $1 \leq i \leq n$.
We will use $(x_{p_i})$ to denote the element of the $G$ (written in its Sylow decomposition) with $x_{p_i}$ as the $i^{th}$ component and identity $0$ everywhere else. 

For $1 \leq i \leq n$, the order of $(x_{p_i})$ is a power of $p_i$ and $H_{p_i}$ is the only component in the decomposition of $G$ that has order divisible by $p_i$, so
$G/\langle x \rangle \cong \bigoplus_{i=1}^n H_{p_i} /\langle (x_{p_i}) \rangle $ and $G/\langle y \rangle \cong \bigoplus_{i=1}^n H_{p_i} /\langle (y_{p_i}) \rangle $. We conclude that
$G/\langle x \rangle \cong G/\langle y \rangle  $ 
is equivalent to \[H_{p_i} /\langle (x_{p_i})\rangle \cong H_{p_i} / \langle (y_{p_i}) \rangle \text{ for all } 1 \leq i \leq n.\]

Combining this with Lemma~\ref{lemmaprod}, because the orders of $H_{p_i}/\langle (x_{p_i})\rangle$ are pairwise coprime for $1 \leq i \leq n$, we see that to prove the sufficient condition for Theorem \ref{mainres}, it is sufficient to prove the condition for $p$-groups. 
\end{proof}

\begin{proof}[Proof of Theorem \ref{mainres}]
    First, we prove that $G/\setx \cong G/\sety$ is a necessary condition for there to exist $\phi \in \Aut(G)$ such that $\phi(x) = y$. 

    Define $X = \langle x \rangle$, $Y = \langle y \rangle$. If $\phi(x) = y$ where $\phi \in \Aut(G)$, we claim that $\psi : a+X \mapsto \phi(a)+Y$ is an isomorphism. First, we show that $\psi$ is well-defined. Assume $a + X = b + X$, then $a-b \in X$. Hence, $\phi(a-b) = \phi(a) - \phi(b) \in Y$ implies that $\phi(a) + Y = \phi(b) + Y$ which proves $\psi(a +X) = \psi(b+X)$. Thus, $\psi$ is well-defined. 
    The surjectivity of $\psi$ follows from the fact that $\phi$ is surjective. Since $x$ and $y$ have the same order, the number of cosets of $X$ in $G$ is equal to the number of cosets of $Y$, so injectivity follows from surjectivity and the fact that $|G|$ is finite. Finally, $\psi$ is a homomorphism as $\psi((a+X)+(b+X)) = \psi((a+b)+X) = \phi(a+b)+Y = (\phi(a)+\phi(b))+Y = (\phi(a)+Y) +  (\phi(b)+Y) = \psi(a+X)+\psi(b+X)$. This implies that $G/\setx \cong G/\sety$. 

Now, we prove that $G/\setx \cong G/\sety$ is sufficient. 

Let $G$ be a finite abelian $p$-group such that
$$
G \cong \bigoplus_{i=1}^nC_{p^{e_i}}
$$
where $p$ is prime and $1 < e_1 \leq e_2\leq  \dots \leq e_n$. 
From~\cite{buzasi} we know that if $G/A \cong G/B$ where $A$ and $B$ are cyclic groups, there exists a $\phi_1 \in \Aut(G)$ that maps the elements of $A$ to the elements of $B$. 
Taking $A = \langle x \rangle$ and $B = \sety$, we have
$$
\phi_1(x) = ky \text{ for some }k \in \mathbb{N}.
$$ 
Therefore, we have that $|ky| = |x| = |y|$. 

We will prove that when $x$ and $y$ are not the identity elements of $G$, we have $\gcd(k, p) = 1$. 
Assume for the sake of contradiction that $p \mid k$. 
Since $G$ is a $p$-group, we know $p \mid |y|$, so
\[
0 = \frac{k}{p} ( |y| y ) =  \frac{|y|}{p}  ( ky)
\]
which implies $|ky| \mid \frac{|y|}{p}$ so $|ky| \neq |y|$, which is a contradiction. Therefore, $\gcd(k, p) = 1$. 

From the above, we know that $k^{-1}$ exists modulo $p^m$ for all $m \in \mathbb{N}$. 
In other words, for each $p^{e_i}$, there exists some $a_i \in \mathbb{N}$ such that $a_i k \equiv 1 \pmod{p^{e_i}}.$

Consider $x = (x_1,x_2,\dots,x_n)$.
We can show that the map defined by $\phi_2 : G \rightarrow G$ with $\phi_2(x) = (a_1x_1, \dots, a_nx_n)$ is an automorphism. Since $\phi_2$ is a linear map, it is also a homomorphism. 
If $\phi_2(g) = \phi_2(h)$, then $k\phi_2(g) = k\phi_2(h)$, implying
$\phi_2(kg) = \phi_2(kh)$, which implies $g = h$. Therefore, $\phi_2$ is injective. Surjectivity follows from injectivity since $G$ is finite. 

From the above, we get that $\phi_2$ is an automorphism so $\phi_2 \circ \phi_1 \in \text{Aut}(G)$.
Then, 
$$(\phi_2 \circ \phi_1)(x) = \phi_2(\phi_1(x)) = \phi_2(ky) = y. $$
Therefore, there exists an automorphism mapping $x$ to $y$, as desired. 
\end{proof}

Theorem \ref{mainres} establishes the necessary and sufficient conditions for two elements in a finite abelian group to be automorphic images of one another. In addition to algorithmic applications, which we discuss in Section \ref{sec:computing}, this theorem can be applied alongside other results in group theory to give us corollaries that would otherwise be difficult to see. 
\begin{thm}
If $x$ and $y$ are both of maximal order in $G$, they are automorphic images of one another. 
\end{thm}
\begin{proof}
Let $G \cong C_{m_1}\oplus C_{m_2} \oplus \dots \oplus C_{m_k}$ be a finite abelian group written in its invariant factor decomposition. Let $C_{m_r}$, $C_{m_{r+1}}$, $\dots$, $C_{m_k}$ be the invariant factors of maximal order. For an element $x \in G$ to be of maximal order, there exists $i$ where $r \leq i \leq k$ such that the $i$th component of $x$ is a generator of $C_{m_i}$. Define this generator as $x_i$. We construct an automorphism $\phi_{x}$ composing the automorphism switching the $i^\text{th}$ and the $k^\text{th}$ components (which is an automorphism because $C_{m_i}$ and $C_{m_k}$ have the same order, the maximal order) with the automorphism mapping where we map $x$ to the element where all components are kept the same except the $k$-th which becomes $1$. Define $x' = \phi_x(x)$; $x'$ must be of the form $x' = (x'_1, x'_2, \dots, x'_{n-1}, 1)$. We construct another map $\psi_x : G \rightarrow G$ as follows: \begin{align*}
\psi_x : (1, 0, \dots, 0, 0) &\mapsto (1, 0, \dots, 0, 0)\\
\psi_x : (0, 1, \dots, 0, 0) &\mapsto (0, 1, \dots, 0, 0)\\
&\;\;\vdots\\
\psi_x : (0, 0, \dots, 1, 0) &\mapsto (0, 0, \dots, 1, 0)\\
\psi_x : (x'_1, x'_2, \dots, x'_{n-1}, 1) &\mapsto (0, 0, \dots, 0, 1).
\end{align*}
We claim $\psi_x$ is an automorphism. First of all, $\psi_x$ is a homomorphism by construction because the preimages of all the maps we defined is a spanning set of $G$. Furthermore, $\psi_x$ is surjective because the images of all the defined maps form a minimal spanning set of $G$. Since $G$ is finite, injectivity is implied, so $\psi_x$ must be an automorphism. 

Similarly, $\phi_y$ can be defined mapping $y$ to $y'$, and $\psi_y$ can be defined mapping $y'$ to $(0, 0, \dots, 0, 1)$. An automorphism mapping $x$ to $y$ is therefore $\phi_y^{-1}\circ\psi_y^{-1}\circ\psi_x\circ\phi_x$. 
\end{proof}
The above theorem combined with Theorem \ref{mainres} implies a result that is not immediately clear. 
\begin{cor}
    Given two elements $x$, $y \in G$ of maximal order, $G/\setx \cong G/\sety$. 
\end{cor}

\section{Computing Automorphic Equivalence of Two Elements}\label{sec:computing}
\subsection{First Algorithm: Directly Applying Smith Normal Form}\label{sec:firstalgo}

It is well known that given finite abelian group $G$ and $x \in G$, $G /\setx$ can be computed by Smith Normal Form  \cite{artin}. Namely, if we let $G \cong C_{m_1} \oplus C_{m_2} \oplus \dots \oplus C_{m_k}$ and $x = (x_1, x_2, \dots, x_k)$. Then, $G/\langle x \rangle$ can be computed by writing  the matrix 
\[ \begin{bmatrix}
x_{1} & x_{2} & \dotsc  & x_{k}\\
m_{1} &  &  & \\
 & m_{2} &  & \\
 &  & \ddots  & \\
 &  &  & m_{k}
\end{bmatrix}\] in Smith Normal Form. 

The time complexity to compute Smith Normal Form for integer matrices in $\mathbb{Z}^{n \times m}$ is $O(n^{\theta-1}mM(n\log(||A||)))$ where $||A|| = \max\{A(i,j) \mid 1 \leq i\leq n, 1 \leq j \leq m\}$, $M(t)$ bounds the cost of multiplying 
two $t$-bit integers, and $\theta$ is the exponent of multiplication of two $n \times n$ matrices ~\cite{Storjohann}. 

The most commonly used fast matrix multiplication algorithm is Strassen's algorithm with a time complexity of $O(n^{2.8074})$ (\cite{strassen}). For a group of rank $n$, the associated matrix to describe in Smith Normal Form is an $(n+1) \times n$ matrix. Assuming multiplication is a constant time operation, this gives a time complexity of verifying $G/\langle x \rangle \cong G/\langle y \rangle$ as $O(n^{2.8074})$ where $n$ is the rank of the group. Therefore, the problem of determining whether two elements are automorphic images of one another can be solved with a time complexity of  $O(n^{2.8074})$.

\subsection{Second Algorithm: Splitting into $p$-groups}\label{sec:secondalgo}

We present another algorithm to compute $G/\setx$. 

To do this, we describe a simpler Smith Normal Form algorithm to compute $G/\langle x \rangle$ for $p$-groups which runs the Smith Normal Form algorithm in \cite{artin} in terms of $\nu_p$. 

\begin{alg}\label{psmith}

Suppose that $G \cong C_{p^{e_1}} \oplus \dots \oplus C_{p^{e_k}}$ and $x = (a_1p^{f_1}, a_2p^{f_2}, \dots, a_kp^{f_k})$ where $p \nmid a_i$ for $1 \leq i 
\leq k$ if $a_i \neq 0$. If any of the $a_i$ are equal to $0$, we can simply remove the zero and remove the component of $G$ that $a_i$ is in. Without loss of generality, let the first component of $x$ be zero, so $x = (0, a_2p^{f_2}, \dots, a_kp^{f_k})$. Consider $x' = (a_2p^{f_2}, \dots, a_kp^{f_k})$ and $G' \cong  C_{p^{e_2}} \oplus \dots \oplus C_{p^{e_k}}$, and note that $G/\langle x \rangle$ is isomorphic to $C_{p^{e_1}} \oplus G'/\langle x' \rangle$. Therefore, assume $a_i \neq 0$ for all $i$.

Without loss of generality, let $f_1 \leq f_2 \leq \dots \leq f_k$ (so the invariant factors of $G$ do not need to be from least to greatest). We only need to consider $x = (p^{f_1}, p^{f_2}, \dots, p^{f_k})$ since it is an automorphic image of $(a_1p^{f_1}, a_2p^{f_2}, \dots, a_kp^{f_k})$. 

First, we write down the following list:

\[\displaystyle \begin{matrix}
f_{1} & f_{2} & \dotsc  & f_{k}\\
e_{1} & e_{2} & \dotsc  & e_{k}. 
\end{matrix}\]

Let $a_{mn}$ denote the element of that list in the $m$th row and the $n$th column. The algorithm is as follows. For each $1 \leq i \leq k-1$, add $\max(0, a_{2i} - a_{1i})$ to all $a_{1j}$ where $i + 1 \leq j \leq k$ and then erase the larger value among $a_{2i}$ and $a_{1i}$. Finally, erase the larger value among $a_{2k}$ and $a_{1k}$. There is one value left per column, which are the powers of the invariant factors of the quotient group. 
\end{alg}

\begin{ex}
As an example, we compute $(C_2\oplus C_4\oplus C_8\oplus C_8)/ \langle (2, 1, 2, 4) \rangle$. Here, $f_1 = 0$, $f_2 = f_3 = 1$, $f_4 = 2$, $e_1 = 2$, $e_2 = 1$, $e_3 = e_4 = 3$. Our list is 

\[\displaystyle \begin{matrix}
0 & 1 & 1 & 2\\
2 & 1 & 3 & 3
\end{matrix} \ .\]
Proceeding with the algorithm, 

\[\begin{matrix}
0 & 3 & 3 & 4\\
 & 1 & 3 & 3
\end{matrix}
\rightarrow 
\begin{matrix}
0 &  & 3 & 4\\
 & 1 & 3 & 3
\end{matrix}
\rightarrow
\begin{matrix}
0 &  & 3 & 4\\
 & 1 &  & 3
\end{matrix}
\rightarrow
\begin{matrix}
0 &  & 3 & \\
 & 1 &  & 3
\end{matrix}
\]
where each stage is the operation run on the succeeding column, so our quotient group is $C_{2^0}\oplus C_{2^1}\oplus C_{2^3}\oplus C_{2^3}$ which is what we would expect if we ran the normal Smith Normal Form algorithm.
\end{ex}

Now, we describe the algorithm for computing $G /\setx$. We break the algorithm into two steps: first, we decompose $G$ into the direct sum of its Sylow $p$-subgroups, and second, we compute $G/\langle x \rangle$ when $G$ is a $p$-group.

Let $G$ be a finite abelian group such that 
$$
G \cong C_{a_1} \oplus C_{a_2} \oplus \cdots \oplus C_{a_n} \cong \bigoplus_{i=1}^t H_{p_i},
$$
where the first representation is its invariant factor decomposition and the $H_{p_i}$ are Sylow $p$-subgroups. 

Decomposing $G$ into a product of $p$-groups requires finding the prime factorization of $a_n$. 

Several algorithms exist to do this; we can use the general number field sieve (see \cite{crandall}), which has heuristic time complexity
\[
O(\exp(((64/9)^{1/3} + o(1)) (\log a_n)^{1/3} (\log \log a_n)^{2/3})).
\]

Let $d(k)$ be the number of prime factors of $k$ (not necessarily distinct). Clearly, \[
d(a_1) \leq d(a_2) \leq \cdots \leq d(a_n) \leq \log_2(a_n).\] For each distinct prime factor $p$ of $a_n$, we can find $\nu_p(a_i)$ for each $i$, which takes \[O(d(a_1) + d(a_2) + \dots + d(a_n)) = O(nd(a_n))\] time overall. This gives us our decomposition into $p$-groups as for each prime $p$, the corresponding $p$-group is 
\[
\bigoplus\limits_{i=1}^n C_{p^{\nu_p(a_i)}} .
\]

The time complexity of this step is $O(n d(a_n))$. We can also do this for $x$ and compute $\nu_p(x_i)$ for each $p\mid a_n$ and $1\leq i \leq n$, which we will use later. This is also $O(n d(a_n))$ since we can perform the exact same algorithm.

Due to Lemma~\ref{lemmaprod}, it is sufficient to find the quotient of each $p$-group component of $G$ by its corresponding $p$-group component of $\setx$. Since this can be done independently for each prime, we instead describe an algorithm to compute $H/\setx$ when $H$ is a $p$-group that runs in $O(k \log k)$ time, where $k$ is the rank of $H$. 

Algorithm \ref{psmith} requires the sequence $f$ to be sorted from least to greatest. We can sort $f$ in $O(k\log k)$ and move around the respective elements in $e$. The remaining algorithm involves the following procedure:
\begin{itemize}
    \item Add some integer to the rest of the elements in the array $f$;
    \item Find the value of $f$ at any position in the list; 
    \item Find the value of $e$ at any position in the list.
\end{itemize}
Notice that querying for a value in $e$ is $O(1)$, as the list is always constant. While doing range add queries on arbitrary intervals and querying a point can be done in $O(k \log k)$ using a Segment Tree, for this specific use case, we can do it in $O(k)$ since we specifically do range add queries on suffixes. We create a variable, call it \textbf{sum}, initialized at 0 storing the amount we need to add to the rest of the array. At each index of $f$, we add \textbf{sum} to the value at $f$. After processing that specific index $i$, we can calculate $\max(0,e_i-f_i)$ and add this value to \textbf{sum} since this is the value we are adding to the rest of the array $f$.

To summarize, given a group $G$ with rank $n$, it takes us $O(nd(a_n))$ to decompose it into its $p$-group components. Each $p$-group component has rank at most $n$ and therefore it takes worst case $O(n\log n)$ to find the quotient group for a $p$-group. There are $d(a_n)$ Sylow $p$-subgroups, so the complexity of computing the quotient groups is $O(n\log(n) d(a_n))$. We can replace all the $d(a_n)$ terms with $\log a_n$, since $d(a_n) \leq \log_2(a_n)$. Adding all the terms together, we get a complexity of 
\[
O(\exp(((64/9)^{1/3} + o(1)) (\log a_n)^{1/3} (\log \log a_n)^{2/3}) + n \log n \log a_n),
\]
which is verified in Appendix A. 

The algorithm is most feasible when $a_n \le 10^{20}$ due to the large complexity contributed by prime factorizing the exponent. 

Now, we compare our two algorithms described above. Although the complexity of the algorithm in Section \ref{sec:secondalgo} is a significant improvement from our algorithm in Section \ref{sec:firstalgo}, it is helpful to know the rank at which the 
former outerforms the latter, since the latter 
has a large constant factor. In this second algorithm, the number of operations required assuming the exponent of the group is $10^{20}$ can be approximated with \[
2\cdot 10^7 + 4 n \cdot 67 + 67 n \log n ,
\]
where the first quantity is from the prime factorization of $a_n$, the second is from computing $\nu_p$ and running the algorithm, and the third is from the sorting.

\begin{figure}[h]
\centering
\includegraphics[scale=0.3]{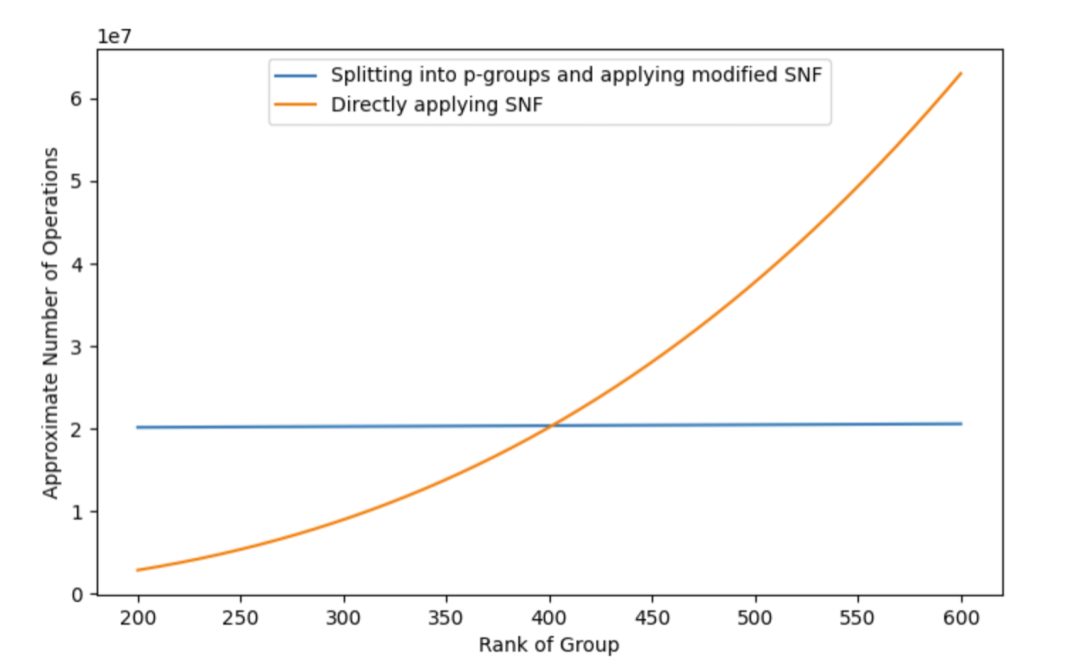}
\caption{Comparing the speed of directly applying SNF versus splitting $G$ into $p$-groups and applying modified SNF on each, assuming the exponent is $10^{20}$.}
\end{figure}

Directly applying Smith Normal Form on $G/\setx$ (the algorithm in Section \ref{sec:firstalgo}) is faster for groups of rank below (approximately) $400$ assuming the exponent of $G$ is $10^{20}$, but for groups of larger ranks, our second algorithm (the algorithm in Section \ref{sec:secondalgo}) is superior. When the exponent of $G$ is small, our algorithm is faster at much smaller ranks. 

\subsection{Comparing Algorithms}
To the best of our knowledge, there is no literature on the time complexity of an algorithm that computes whether two elements of a group are automorphic images of one another. Theorem~\ref{mainres}, combined with efficient algorithms to determine whether two quotient groups are isomorphic, provides a strategy to implement an efficient algorithm for this purpose. It is more reasonable to compute the Smith Normal Form for integer matrices to check if $G/\langle x \rangle \cong G/\langle y \rangle$ than determine an automorphism that maps $x$ to $y$. 

Naively, a brute force algorithm to determine if $x$ and $y$ are automorphic images would be to compute $\Aut(G)$ and then iterate through it, computing $\phi(x)$ for all $\phi \in \Aut(G)$ and checking if it is equal to $y$. In the worst-case scenario, where $x$ and $y$ are not automorphic images of each other, we must iterate through the entirety of $\Aut(G)$.
The best known algorithm to compute $\Aut(G)$ implemented in GAP is described by 
Eick, Leedham-Green and O'Brien. The time complexity of this algorithm is dominated by $n^7$ where $n$ is the rank of $G$ (see Section 11.1 in \cite{eick}). If we run this algorithm and then iterate through all elements of $\Aut(G)$ to check if $\phi(x) = y$, we obtain a time complexity of $O(|\Aut(G)| + n^7)$. 
Our algorithm compares very favorably at $O(n^{2.8})$.
Furthermore, we can show that $\Aut(G)$ is exponential in $n$, where $n$ is the rank of $G$. 
\begin{prop}\label{lem:sizeautg}
    $|\Aut(G)|$ is exponential in $n$, where $n$ is the rank of $G$. 
\end{prop}
\begin{proof}
From Lemma \ref{lemmaprod}, we may assume that $G$ is a $p$-group.
Let $G \cong \bigoplus_{i=1}^k  C_{p^{m_i}}^{n_i}$ and let $n=n_1+n_2+\dots+n_k$ be the rank of $G$.
As a corollary to Theorem 4.1 in \cite{hillar}, we get $|\Aut(C_{p^{m_i}}^{n_i})| = {p^{(m_i-1){n_i}^2}}\prod\limits_{j=0}^{n_i-1}(p^{n_i}-p^j)  = p^{{m_i}{n_i}^2}\prod\limits_{j=0}^{n_i-1}\left(1-\dfrac{1}{p^{n_i-j}}\right)$. 
Since $p \geq 2$, this gives $|\Aut(C_{p^{m_i}}^{n_i})| \geq \dfrac{p^{{m_i}{n_i}^2}}{2^{n_i}}$.
For $p=2$, we have
$|\Aut(C_{p^{m_i}}^{n_i})| \geq \dfrac{p^{{m_i}{n_i}^2}}{2^{n_i}} = 2^{n_i(m_in_i-1)} > 2^{n_i}$.
For $p > 3$, $|\Aut(C_{p^{m_i}}^{n_i})| \geq \dfrac{p^{{m_i}{n_i}^2}}{2^{n_i}} > \left( \dfrac{p}{2}\right)^{n_i} $. 
Considering elements of $G$ to be $k$-tuples, we can construct an automorphism in $G$ by individual component-wise automorphisms of $C_{p^{m_i}}^{n_i} $.
This gives $|\Aut(G)| \geq \prod\limits_{i=1}^k  |\Aut(C_{p^{m_i}}^{n_i})|$.
For $p=2$, this implies that $|\Aut(G)| >  \prod\limits_{i=1}^k  2^{n_i} = 2^n$. For $p > 2$, this implies that 
$|\Aut(G)| \geq  \prod\limits_{i=1}^k  \left( \dfrac{p}{2}\right)^{n_i}  =  \left( \dfrac{p}{2}\right)^n$. 
In either case, we prove that $|\Aut(G)|$ is exponential in the rank of $G$.
\end{proof}
Since we have shown that $|\Aut(G)|$ is exponential in the rank $n$, the runtime of the naive algorithm, which has time complexity of  $O(|\Aut(G)| + n^7)$, is 
 exponential in $n$. This is significantly worse than our runtime of $O(n\log (n) \log a_n)$, especially for groups with small exponent. See Appendix A for an implementation of Algorithm \ref{psmith} and its numerical runtime analysis. 


\section{Computing Automorphic Orbits}
Consider a finite abelian group $G$ with invariant factor decomposition $C_{d_1} \oplus C_{d_2} \oplus \cdots \oplus C_{d_n}.$ We describe an algorithm that computes the automorphic orbits in $G$ with a time complexity of $O(\sqrt{|G|}2^n  n \log n)$. However, this algorithm works significantly faster for most groups $G$ and achieves this complexity only when the invariant factors of $G$ are relatively small (i.e. less than or equal to $100$). 

Previous works have computed the number of orbits in finite abelian groups (see \cite{SS,dutta}). Similar to \cite{rocca2011automorphismclasseselementsfinitely}, we compute the orbits (hence also the number of orbits) by reducing each element to a specific form and applying Algorithm Two to compute the quotient group, which can be used to determine which orbit each element is in. 

Define the orbits of $G$  the equivalence classes by $\sim$, where $x \sim y$ if and only if there exists $\phi \in \Aut(G)$ such that $\phi(x) = y$. Equivalently, we can define them as the orbits of the natural action of $\Aut(G)$ on $G$. 

By the same process described earlier (in the algorithm in Section 3.2), we can split $G$ into a direct sum of $p$-groups. The time complexity of prime factorization is at most $O(\sqrt{\text{exp}(G)})$, and since the latter steps require a complexity greater than $O(\sqrt{|G|})$, we may ignore the prime factorization step in the final complexity. 

We first describe the algorithm for finding automorphic orbits in $p$-groups and show how these orbits can subsequently be combined to find the orbits of $G$.

Let $H$ be a $p$-group of the form $
\bigoplus\limits_{i=1}^n C_{p^{e_i}}$. 
For each pair of elements $x, y$ in the orbit $\mathcal O$ of $H$, recall that $H /\langle x \rangle \cong H/\langle y \rangle$ by Theorem \ref{mainres}. Therefore, we can say a group $K$ 
\textit{corresponds} to orbit $\mathcal O$ if $H/\langle x\rangle \cong K$ for each $x\in \mathcal O$.

Consider an element $x \in H$ of the form $(a_1 p^{b_1}, a_2 p^{b_2}, \cdots,a_n  p^{b_n})$, where $0\leq b_i\leq e_i$ and $p \nmid a_i$ for each $1\leq i \leq n$. We can compute this representation in $O(n \log(\text{exp}(H)))$ by computing $\nu_p$ for each component of $H$. Again, this step is insignificant compared to the complexity of future steps and can be ignored. The element $x$ is in the same orbit as $(p^{b_1}, p^{b_2}, \cdots, p^{b_n})$ because there exists an automorphism between $x$ and $(p^{b_1}, p^{b_2}, \cdots, p^{b_n})$ component-wise for each component. We say $(p^{b_1}, p^{b_2}, \cdots, p^{b_n})$ is the \textit{reduced form} of $x$. Therefore, it is sufficient to compute the orbits of the elements that satisfy $a_1=a_2 = \cdots = a_n = 1$. 

There are $\prod_{i=1}^n (e_i+1)$ such elements. Therefore, for each element $x = (p^{b_1}, p^{b_2}, \cdots, p^{b_n})$, we can compute $H/\langle x\rangle$, and any element that has a reduced form equal to $x$ is in the orbit corresponding to $H/\langle x \rangle$. In fact, we can compute the number of elements that have this property. Since this is independent for each component, we show how to calculate it for the $i$th component and multiply this across. 

If $b_i=e_i$, then there is only one distinct $x$ as the component is just 0. If $b_i < e_i$, then we claim there are $p^{e_i-b_i} - p^{e_i-b_i-1}$. It is sufficient to compute how many $1\leq a_i\leq p^{e_i-b_i}$ exist such that $\gcd(a_i,p)=1$, but this is simply $\phi(p^{e_i-b_i})$,` where $\phi$ is the Euler phi function. 

Overall, this algorithm takes $O( \prod (e_i+1) \cdot n \log n)$ time, and it computes the size of each orbit and the reduced forms that are part of that orbit. 

Let $k$ be the number of distinct prime factors of the exponent of $G$. Performing this algorithm for all $k$ different $p$-groups gives $k$ sets of orbits $S_1, S_2, \dots, S_k$. The number of orbits is $|S_1| |S_2| \cdots |S_k|$, since we choose orbits $\mathcal{O}_1 \in S_1, \mathcal{O}_2 \in S_2, \cdots, \mathcal{O}_k\in S_k$ and each of these $k$-tuples of orbits corresponds to a unique orbit in $G$ due to Lemma \ref{lemmaprod}.

We can also find the properties of this unique orbit. The size of the orbit is the product of the sizes of the $k$ individual orbits. The representative elements of the orbit are any combination of $k$ representative elements, one from each $\mathcal{O}_i$. In other words, $x\in G$ is a representative element for this orbit if and only if the component of $x$ for the $i$th Sylow $p$-subgroup of $G$ is one of the representatives for $\mathcal{O}_i$.

Overall, the time complexity is \[O\left(\prod_{\substack{p \leq n\\ \text{$p$ prime}}} \prod_{i=1}^n (\nu_p(d_i) + 1)  \cdot n \log n\right).\] Notice that the double product simply computes the number of factors of $d_i$, so we can rewrite this as 
\[
O\left( \prod_{i=1}^n \tau(d_i) n \log n\right).
\]
However, $\tau(d_i) \leq 2\sqrt{d_i}$, so we can bound this complexity above with 
\[
O\left( \prod_{i=1}^n (2\sqrt{d_i}) n \log n\right) = O(\sqrt{|G|} 2^n \cdot n \log n).
\]
As stated earlier, when the $d_i$ are large, the $\sqrt{d_i}$ replacement is weak and the algorithm is much faster.

\appendix
\newpage
\section{Pseudocode for Algorithms and Numerical Analysis}
The following is pseudocode for Algorithm \ref{psmith}, where the isAutoImage($x$, $y$, $G$) function returns whether $x$ and $y$ are automorphic images in finite abelian group $G$ written in its invariant factor decomposition. 
\newline
\begin{breakablealgorithm}
\begin{algorithmic}[1]
\Function{DecomposeElementIntoParts}{$G$, $x$}
    \State $\textit{factors} \gets \textit{factors}(|G|)$
    \For{$p$ \textbf{in} \textit{factors}}
        \For{$i \gets 1$ \textbf{to} $|G|$}
            \If{$\nu_p(G[i]) \neq 0$}
                \State Add $i$ to \textit{psubindeces}
            \EndIf
        \EndFor
        \For{$i$ \textbf{in} \textit{psubindeces}}
            \State Add $\nu_p(x[i])$ to \textit{components}
        \EndFor
        \State Add \textit{components} to \textit{pParts}
    \EndFor
    \State \Return \textit{pParts}
\EndFunction
\Function{DecomposeAbelianGroup}{$G$}
    \State $\textit{factors} \gets \textit{factors}(|G|)$
    \For{$p$ \textbf{in} \textit{factors}}
        \For{$i \gets 1$ \textbf{to} $|G|$}
            \If{$\nu_p(G[i]) \neq 0$}
                \State Add $\nu_p(G[i])$ to \textit{components}
            \EndIf
        \EndFor
        \State Add \textit{components} to \textit{pParts}
    \EndFor
    \State \Return \textit{pParts}
\EndFunction
\Function{Sort}{$f$}
    \State Sort $f$ by the first element in the pair
\EndFunction
\Function{pGroupSNF}{$f$}
    \State \Call{Sort}{$f$}
    \For{$i \gets 1$ \textbf{to} \Call{Length}{$f$}}
        \State Add \textit{addto} to $f[i][1]$
        \State Add \Call{Max}{$0, f[i][2]-f[i][1]$} to \textit{addto}
    \EndFor
    \For{$i \gets 1$ \textbf{to} \Call{Length}{$f$}}
        \State Add \Call{Min}{$f[i][1], f[i][2]$} to \textit{final}
    \EndFor
    \State \Return \textit{final}
\EndFunction
\Function{combineLists}{\textit{list1}, \textit{list2}}
    \For{$i \gets 1$ \textit{to} \Call{length}{$list1$}}
        \State \textit{list}[i] $\gets$ [\textit{list1}[i], \textit{list2}[i]]
    \EndFor
    \State \Return \textit{list}
\EndFunction
\Function{isAutoImage}{$x$, $y$, $G$}
    \State \textit{decomp} $\gets$ \Call{DecomposeAbelianGroup}{$G$}
    \State \textit{partsx} $\gets$ \Call{DecomposeElementsIntoParts}{$x$}
    \State \textit{manipx} $\gets$ \Call{CombineLists}{\textit{partsx, decomp}}
    \For{$i \gets 1$ \textbf{to} \Call{Length}{\textit{manipx}}}
        \State \textit{manipx[i]} $\gets$ \Call{pGroupSNF}{\textit{manipx}[i]}
    \EndFor
    \State \textit{partsy} $\gets$ \Call{DecomposeElementsIntoParts}{$y$}
    \State \textit{manipy} $\gets$ \Call{CombineLists}{\textit{partsy, decomp}}
    \For{$i \gets 1$ \textbf{to} \Call{Length}{\textit{manipy}}}
        \State \textit{manipy}[i] $\gets$ \Call{pGroupSNF}{\textit{manipy}[i]}
    \EndFor
    \If{\textit{manipx} = \textit{manipy}}
        \State \Return true
    \EndIf
    \State \Return false
\EndFunction
\end{algorithmic}
\end{breakablealgorithm}

We ran this algorithm for groups of the form $C_4^n$ for $n \in \{3+10k \mid 0 \leq k \leq 16\} \cup \{2^k \mid 1 \leq k \leq 9\}$, $x 
= (1, 1, \dots)$, and $y = (3, 3, \dots)$ so that the exponent remains constant and small enough so that the prime factorization does not contribute significantly to the runtime. The data is shown in the table below. 
\begin{table}[h!]
    \centering
    \begin{tabular}{|c|*{14}{c|}}
        \hline
        \textbf{Rank} & 2 & 3 & 4 & 8 & 13 & 16 & 23 & 32 & 33 & 43 & 53 & 63 & 64 & 73 \\
        \hline
        \textbf{Runtime (ms)} & 1.6 & 2 & 1.8 & 3.6 & 4 & 4.6 & 8 & 9.2 & 9 & 17 & 20 & 26 & 26.6 & 35 \\
        \hline
    \end{tabular}
    \caption{Rank vs Runtime (ms)}
\end{table}

\begin{table}[h!]
    \centering
    \begin{tabular}{|c|*{12}{c|}}
        \hline
        \textbf{Rank} & 83 & 93 & 103 & 113 & 123 & 128 & 133 & 143 & 153 & 163 & 256 & 512 \\
        \hline
        \textbf{Runtime (ms)} & 44 & 61 & 72 & 88 & 106 & 110.6 & 122 & 145 & 173 & 199 & 545.8 & 3263 \\
        \hline
    \end{tabular}
    \caption{Rank and Runtime (ms) continued}
\end{table}
Implementing in GAP and using Python polynomial fitting code, the best fit polynomial is $0.2129476474670508x^{1.28247480729629}$, which is near-linear and consistent with the runtime of $O(n\log n)$.

\section*{Acknowledgements}
The first three authors would like to express their heartfelt gratitude to their mentor Jared Kettinger, the fourth author, for his support and insight during the entire research process. All four authors would like to thank their mentor, Professor Coykendall, for his knowledge and expertise. The authors also kindly thank Dr. Felix Gotti and the PRIMES-USA research program for giving them this amazing opportunity to learn and conduct research. 


\bibliographystyle{unsrt}
\bibliography{References}











\end{document}